\documentclass[final,11pt]{elsarticle}

\usepackage{verbatim,a4wide}

\usepackage{amsmath}
\usepackage{amssymb}
\usepackage{amsthm}
\usepackage{bm}
\usepackage{algorithm}
\usepackage[noend]{algpseudocode}

\usepackage[cp1250]{inputenc}
\usepackage[OT4]{fontenc}
\usepackage[english]{babel}

\numberwithin{equation}{section}

\newtheorem{theorem}{Theorem}[section]

\newdefinition{remark}[theorem]{Remark}
\newdefinition{corollary}[theorem]{Corollary}
\newdefinition{definition}[theorem]{Definition}
\newdefinition{problem}[theorem]{Problem}
\newdefinition{example}[theorem]{Example}

\journal{Elsevier} 

\begin{document}

\begin{frontmatter}

\title{Fast and accurate evaluation of dual Bernstein polynomials}

\author{Filip Chudy}
\ead{Filip.Chudy@cs.uni.wroc.pl}

\author{Pawe{\l} Wo\'{z}ny\corref{cor}}
\ead{Pawel.Wozny@cs.uni.wroc.pl}
\cortext[cor]{Corresponding author. Fax {+}48 71 3757801}
\address{Institute of Computer Science, University of Wroc{\l}aw,
         ul.~Joliot-Curie 15, 50-383 Wroc{\l}aw, Poland}

\begin{abstract}
Dual Bernstein polynomials find many applications in approximation theory,
computational mathematics, numerical analysis and computer-aided geometric
design. In this context, one of the main problems is fast and accurate 
evaluation both of these polynomials and their linear combinations. New simple
recurrence relations of low order satisfied by dual Bernstein polynomials are
given. In particular, a first-order non-homogeneous recurrence relation linking
dual Bernstein and shifted Jacobi orthogonal polynomials has been obtained. 
When used properly, it allows to propose fast and numerically efficient
algorithms for evaluating all $n+1$ dual Bernstein polynomials of degree $n$ 
with $O(n)$ computational complexity.
\end{abstract}


\end{frontmatter}

\section{Introduction}                                  \label{S:Introduction}

Let us introduce the inner product 
$\left<\cdot,\cdot\right>_{\alpha,\beta}$ by
\begin{equation}\label{E:InnerProd}
\left<f,g\right>_{\alpha,\beta}:=
             \int_{0}^{1}(1-x)^\alpha x^\beta f(x)g(x)\,\mbox{d}x,
\end{equation}
where $\alpha,\beta>-1$.
             
Let $\Pi_n$ $(n\in\mathbb N)$ denote the set of polynomials of degree at most
$n$. Recall that the \textit{shifted Jacobi polynomial of degree $n$},
$R_n^{(\alpha,\beta)}\in\Pi_n$, is defined by
\begin{equation}\label{E:JacobiP}
R_n^{(\alpha,\beta)}(x):=
       \frac{(\alpha+1)_n}{n!}
          \sum_{k=0}^{n}\frac{(-n)_k(n+\alpha+\beta+1)_k}
                             {k!(\alpha+1)_k}(1-x)^k \qquad(n=0,1,\ldots).
\end{equation}
Here $(c)_l$ $(c\in\mathbb C;\ l\in\mathbb N)$ denotes the \textit{Pochhammer
symbol},
$$
(c)_0:=1,\qquad (c)_l:=c(c+1)\ldots(c+l-1)\quad (l\geq 1).
$$             
These polynomials satisfy the second-order recurrence relation of the form
\begin{equation}\label{E:JacobiRecRel}
\xi_0(n)R^{(\alpha,\beta)}_n(x)+
\xi_1(n)R^{(\alpha,\beta)}_{n+1}(x)+\xi_2(n)R^{(\alpha,\beta)}_{n+2}(x)=0
\qquad (n=0,1,\ldots),
\end{equation}
where  
\begin{eqnarray}
&&\label{E:Def_xi_0}
\xi_0(n):=-2(n+\alpha+1)(n+\beta+1)(2n+\sigma+3),\\
&&\label{E:Def_xi_1}
\xi_1(n):=(2n+\sigma+2)
                \{(2n+\sigma+1)(2n+\sigma+3)(2x-1)+\alpha^2-\beta^2\},\\
&&\label{E:Def_xi_2}
\xi_2(n):=-2(n+2)(n+\sigma+1)(2n+\sigma+1),
\end{eqnarray}
and $\sigma:=\alpha+\beta+1$ (cf., e.g., \cite[\S1.8]{KS1998}).

\begin{remark}\label{R:JacobiRecRel}
Notice that the recurrence relation~\eqref{E:JacobiRecRel} can be used, for 
example, in fast and accurate methods for evaluating the values 
$R^{(\alpha,\beta)}_n(x)$ for a given $x,\alpha,\beta$ and all $0\leq n\leq N$,
where $N$ is a~fixed natural number, with $O(N)$ computational complexity. For
more details about performing computations with recurrence relations properly,
see \cite{Wimp1984}. 
\end{remark}

Shifted Jacobi polynomials are orthogonal with respect to the inner
product~\eqref{E:InnerProd}, i.e., 
$$
\left<R^{(\alpha,\beta)}_k,R^{(\alpha,\beta)}_l\right>_{\alpha,\beta}=
\delta_{kl}h_k\qquad (k,l\in\mathbb N),
$$
where $\delta_{kl}$ is the \textit{Kronecker delta} ($\delta_{kl}=0$ for 
$k\neq l$ and $\delta_{kk}=1$) and
$$
h_k:=K\,\frac{(\alpha+1)_k(\beta+1)_k}
             {k!(2k/\sigma+1)(\sigma)_{k}}\qquad (k=0,1,\ldots)
$$
with $K:={\Gamma(\alpha+1)\Gamma(\beta+1)}/{\Gamma(\sigma+1)}$. For more
properties and applications of polynomials $R^{(\alpha,\beta)}_n$, see, e.g.,
\cite{KS1998,AAR1999}.

Let $B^n_0,B^n_1,\ldots, B^n_n\in\Pi_n$ be \textit{Bernstein basis polynomials}
given by
\begin{equation}\label{E:BernPoly}
B^n_i(x):=\binom{n}{i}x^i(1-x)^{n-i}\qquad (i=0,1,\ldots,n;\ n\in\mathbb N).
\end{equation}

\begin{definition}[{\cite[\S5]{LW2006}}]\label{D:DualBer}
\textit{Dual Bernstein polynomials of degree $n$},
\begin{equation}\label{E:DualBerPoly}
D^n_0(x;\alpha,\beta),\, D^n_1(x;\alpha,\beta),\,\ldots,\,
D^n_n(x;\alpha,\beta)\in\Pi_n,
\end{equation}
are defined so that the following conditions hold:
$$
\left<B^n_i,D^n_j(\cdot;\alpha,\beta)\right>_{\alpha,\beta}=\delta_{ij}
\qquad (i,j=0,1,\ldots,n)
$$
(cf.~\eqref{E:InnerProd}). We adopt the convention that 
$D^n_i(x;\alpha,\beta):=0$ for $i<0$ or $i>n$.
\end{definition}

Let us mention that in the case $\alpha=\beta=0$ these polynomials were 
introduced earlier by Ciesielski in~\cite{ZC1987}.

Certainly, $\mbox{lin}\{B^n_k\;:\;0\leq k\leq n\}=\Pi_n$. One can also prove 
that dual Bernstein polynomials~\eqref{E:DualBerPoly} form a basis of the 
$\Pi_n$ space.  

It is well-known that for many years Bernstein basis polynomials have been used 
in computer-aided geometric design, approximation theory, numerical analysis 
and computational mathematics. See, e.g., books \cite{Bustamante2017,Farin2002}
and article \cite{Farouki2012}, as well as papers cited 
therein.

For a given function $f$, let us define a polynomial $p^\ast_n$ of the following
\textit{Bernstein-B\'{e}zier form}:
$$
p^\ast_n(x):=\sum_{k=0}^{n}I_kB^n_k(x),
$$
where
\begin{equation}\label{E:IntI_k} 
I_k:=\left<f,D^n_k(\;\cdot\;;\alpha,\beta)\right>_{\alpha,\beta}=
    \int_{0}^{1}(1-x)^\alpha x^\beta f(x)D^n_k(x;\alpha,\beta)\,\mbox{d}x
                                                       \qquad (0\leq k\leq n).
\end{equation}
Recall that the polynomial $p^\ast_n$ minimizes the value of the least-square
error
$$
||f-p_n||_2^2:=\left<f-p_n,f-p_n\right>_{\alpha,\beta}=
      \int_{0}^{1}(1-x)^\alpha x^\beta (f(x)-p_n(x))^2\mbox{d}x
                                                    \qquad (p_n\in\Pi_n)
$$
(cf.~\cite[Lemma 2.2]{WGL2015}).

This is one of the main reasons that dual Bernstein polynomials have recently
been extensively studied and found many theoretical (see
\cite{ZC1987,BJ1998,LW2006,RN2007,RN2008,LW2011}) and practical applications.
For example, these dual polynomials are very useful in: curve intersection 
using B\'{e}zier clipping (\cite{SN1990,BJ2007,LZLW2009}); degree reduction and
merging of B\'{e}zier curves (\cite{WL2009,GLW2016,GLW2017,WGL2015}); 
polynomial approximation of rational B\'{e}zier curves (\cite{LWK2012});
numerical solving of boundary value problems (\cite{GW2018}) or even fractional
partial differential equations (\cite{JBJ2017,JJBB2017}). Skillful use of these
polynomials often results in less costly algorithms of solving many 
computational problems.

In some of the mentioned tasks, as well as in finding the solution of the
least-square problem in the Bernstein-B\'{e}zier form, it is necessary 
to compute the numerical approximations of the collection of integrals
\eqref{E:IntI_k} for all $k=0,1,\ldots,n$ and a given function $f$. In general,
to do so, one has to use quadrature rules (see, e.g., \cite[\S5]{DB}), but it
requires fast evaluation of all $n+1$ dual Bernstein polynomials of degree $n$ 
in many \textit{nodes}. Thus, the authors consider the following problem.

\begin{problem}[cf.~\cite{ChW2018}]\label{P:Problem1}
Let us fix numbers: $n\in\mathbb N$, $x\in\mathbb [0,1]$ and $\alpha,\beta>-1$.
Compute the values 
$$
D^n_i(x;\alpha,\beta)
$$
for all $i=0,1,\ldots,n$. 
\end{problem}

Using new differential-recurrence properties of polynomials
$D^n_i(x;\alpha,\beta)$ obtained in~\cite{ChW2018}, the authors have recently
constructed the fourth-order recurrence relation of the following form:
\begin{equation}\label{E:DualBer-Rec-Rel-b}
\sum_{j=-2}^{2}v_j(i)D^n_{i+j}(x;\alpha,\beta)=0
\end{equation}
satisfied by dual Bernstein polynomials, where coefficients $v_j$ 
$(-2\leq j\leq 2)$ are quintic polynomials in $i$. 
See~\cite[Corollary 4.2 and Eq.~(4.2)]{ChW2018}. Notice that this result follows
from relations between dual Bernstein and shifted Jacobi polynomials
(cf.~\eqref{E:JacobiP}), as well as so-called Hahn orthogonal polynomials (see,
e.g., \cite[\S1.5]{KS1998}). 

The recurrence relation \eqref{E:DualBer-Rec-Rel-b} has been used to propose an
algorithm which solves Problem~\ref{P:Problem1} with $O(n)$ computational
complexity. Notice that previously known methods have $O(n^2)$ or even $O(n^3)$
computational complexity. Experiments have shown that the new method is much
faster and gives good numerical results for low $n$ $(n\approx 20,30)$. 
See~\cite[\S6]{ChW2018}.

The first goal of this paper is to derive new recurrence relations of lower 
order for dual Bernstein polynomials. More specifically, in Section
\ref{SS:NH-RecRel-1}, a simple first-order non-homogeneous recurrence relation
for dual Bernstein polynomials is given. Next, in Section~\ref{SS:RecRel-2-3}, 
we find homogeneous recurrence relations of the second and third order for
polynomials $D^n_i(x;\alpha,\beta)$.

The second goal is to propose fast and accurate algorithms which solve
Problem~\ref{P:Problem1} and have $O(n)$ computational complexity, as well as
work even for large values of $n$ $(n\approx1000,2000)$ (cf.~\cite{Bezerra2013},
where evaluation of high-degree polynomials in Bernstein-B\'{e}zier form was
examined, or signal processing, as well as numerical integration of
highly–oscillating functions, where high-degree polynomials may appear). 
Such efficient methods---based on the relation obtained 
in~\S\ref{SS:NH-RecRel-1}---are presented in Section~\ref{S:DB-Eval}. Results
of numerical experiments are given in~\S\ref{S:NumExperiments}.

\section{New recurrence relations}                         \label{S:NewRecRel}

Now, let us recall some properties of dual Bernstein polynomials
(\ref{E:DualBerPoly}) and shifted Jacobi polynomials (\ref{E:JacobiP}) which
allow us to derive new recurrence relations for $D^n_i(x;\alpha,\beta)$.

In \cite[Theorem 5.1]{LW2006}, the following relation between dual Bernstein
polynomials of degrees $n$, $n+1$, as well as the shifted Jacobi polynomial of
degree $n+1$ has been proven:
\begin{equation}\label{E:RecRel-I}
D^{n+1}_i(x;\alpha,\beta)=\left(1-\dfrac{i}{n+1}\right)
   \,D^{n}_{i}(x;\alpha,\beta)+\frac{i}{n+1}\,D^{n}_{i-1}(x;\alpha,\beta)+
                         C^{(\alpha,\beta)}_{ni}R^{(\alpha,\beta)}_{n+1}(x),
\end{equation}
where $0\le i\le n+1$, and 
\begin{equation}\label{E:Def_C}
C^{(\alpha,\beta)}_{ni}:=(-1)^{n-i+1}\frac{(2n+\sigma+2)(\sigma+1)_n}
                             {K(\alpha+1)_{n-i+1}(\beta+1)_i}.
\end{equation}
Note that for $\alpha=\beta=0$, this identity was found earlier by Ciesielski
in~\cite{ZC1987}.   

It is also known that dual Bernstein polynomials satisfy a \textit{symmetry
relation} of the type
\begin{equation}\label{E:DualBerSym}
D^n_i(x;\alpha,\beta)=D^n_{n-i}(1-x;\beta,\alpha)\qquad (i=0,1,\ldots,n).
\end{equation}
See \cite[Corollary 5.3]{LW2006}.

The polynomial $D^n_i(x;\alpha,\beta)$ can be expressed as \textit{a short}
linear combination of $\min(i,n-i)+1$ shifted Jacobi polynomials with shifted
parameters: 
\begin{eqnarray*}
D^n_i(x;\alpha,\beta)&=&\frac{(-1)^{n-i}(\sigma+1)_{n}}
                             {K\,(\alpha+1)_{n-i}(\beta+1)_i}
                                  \sum_{k=0}^{i}\frac{(-i)_k}{(-n)_k}\,
				                           R^{(\alpha,\beta+k+1)}_{n-k}(x),\\				         
D^n_{n-i}(x;\alpha,\beta)&=&\frac{(-1)^{i}(\sigma+1)_{n}}
                                 {K\,(\alpha+1)_{i}(\beta+1)_{n-i}}
                                    \sum_{k=0}^{i}(-1)^k\frac{(-i)_k}{(-n)_k}
                                             \,R^{(\alpha+k+1,\beta)}_{n-k}(x),
\end{eqnarray*}
where $i=0,1,\ldots,n$. See~\cite[Corollary 5.4]{LW2006}. In particular, we 
have
\begin{eqnarray}
&&\label{E:D^n_0}
D^n_0(x;\alpha,\beta)=\frac{(-1)^{n}(\sigma+1)_{n}}
                           {K\,(\alpha+1)_n}R^{(\alpha,\beta+1)}_{n}(x),\\
&&\label{E:D^n_n}
D^n_n(x;\alpha,\beta)=\frac{(\sigma+1)_{n}}
                           {K\,(\beta+1)_n}R^{(\alpha+1,\beta)}_{n}(x).
\end{eqnarray}

Using \cite[Eq.~(3.1)]{ChW2018}, so-called \textit{Chu-Vandermonde identity}
(see, e.g., \cite[Corollary 2.3.]{AAR1999}) and symmetry~\eqref{E:DualBerSym},
one can check that
\begin{eqnarray}
&&\label{E:DualBer-1}
D^n_i(1;\alpha,\beta)=(-1)^{n-i}\frac{(\sigma+1)_{n}(n-i+\alpha+2)_{i}}
                                     {K\,n!(\beta+1)_i},\\
&&\label{E:DualBer-0}                                     
D^n_i(0;\alpha,\beta)=(-1)^{i}\frac{(\sigma+1)_{n}(i+\beta+2)_{n-i}}
                                   {K\,n!(\alpha+1)_{n-i}}
                                                       \qquad (0\leq i\leq n).
\end{eqnarray}

From \cite[Eq.~(6.4.20) and (6.4.23)]{AAR1999}, it follows that
\begin{eqnarray}
\label{E:JacobiConnection-1}
(1-x)R_n^{(\alpha+1,\beta)}(x)&=& 
                      -\frac{n+1}
                            {2n+\sigma+1}R_{n+1}^{(\alpha,\beta)}(x)+ 
                       \frac{n+\alpha+1}
                            {2n+\sigma+1}R_n^{(\alpha,\beta)}(x),\\
\label{E:JacobiConnection-2}
xR_n^{(\alpha,\beta+1)}(x)&=&
                     \frac{n+1}
                          {2n+\sigma+1}R_{n+1}^{(\alpha,\beta)}(x)+ 
                     \frac{n+\beta+1}
                          {2n+\sigma+1}R_n^{(\alpha,\beta)}(x).
\end{eqnarray}

\subsection{First-order non-homogeneous recurrence relation}    
                                                        \label{SS:NH-RecRel-1}

Using the results mentioned above and relation~\eqref{E:JacobiRecRel}, one can
justify a simple first-order non-homogeneous recurrence relation for dual
Bernstein polynomials.   

\begin{theorem}\label{T:NH-Rec-Rel-1}
For $i=0,1,\ldots,n$, the following relation holds:
\begin{equation}\label{E:NH-Rec-Rel-1}
(x-1)(i+1)D^n_i(x;\alpha,\beta)+x(n-i)D^n_{i+1}(x;\alpha,\beta)=
     \frac{-C^{(\alpha,\beta)}_{n,i+1}}{2n+\sigma+2}T^{(\alpha,\beta)}_{ni}(x),
\end{equation}
where the notation used is that of~\eqref{E:Def_C}, and
\begin{equation}\label{E:Def_T_I}
T^{(\alpha,\beta)}_{ni}(x):=(n-i)(n+\alpha+1)xR^{(\alpha,\beta+1)}_n(x)+
                            (i+1)(n+\beta+1)(1-x)R^{(\alpha+1,\beta)}_n(x).
\end{equation} 
\end{theorem}
\begin{proof}
In the sequel, we need the following identity
\begin{multline}
\frac{2n+\sigma+1}{n+1}T^{(\alpha,\beta)}_{ni}(x)=
                (n+\alpha+1)(n+\beta+1)R_{n}^{(\alpha,\beta)}(x)\\
\label{E:Def_T_II}
+\Big((n-i)(n+\alpha+1)-(i+1)(n+\beta+1)\Big)R_{n+1}^{(\alpha,\beta)}(x),
\end{multline}
which can be verified using~\eqref{E:JacobiConnection-1} 
and~\eqref{E:JacobiConnection-2}. 

Let us use induction on $n$. First, observe that for any $i=n$, 
the relation~\eqref{E:NH-Rec-Rel-1} immediately follows from~\eqref{E:D^n_n}. 
So, in particular, it also holds for $n=0$.

Now, suppose that~\eqref{E:NH-Rec-Rel-1} is true for some natural number $n$
and all $0\leq i\leq n$. One has to prove that
$$
(x-1)(i+1)D^{n+1}_i(x;\alpha,\beta)+x(n-i+1)D^{n+1}_{i+1}(x;\alpha,\beta)+
\frac{C^{(\alpha,\beta)}_{n+1,i+1}}
                   {2n+\sigma+4}T^{(\alpha,\beta)}_{n+1,i}(x)\equiv0,
$$
where $0\leq i\leq n+1$. We already know that it holds for $i=n+1$. Assume that
$0\leq i\leq n$. Applying twice~\eqref{E:RecRel-I} to the left-hand side,
using~\eqref{E:Def_T_II} and doing simple algebra, one can obtain its equivalent
form
\begin{multline*}
\frac{n-i+1}{n+1}\Big[(x-1)(i+1)D^n_i(x;\alpha, \beta)+
                                     x(n-i)D^n_{i+1}(x;\alpha, \beta)\Big]\\
+\frac{i+1}{n+1}\Big[(x-1)iD^n_{i-1}(x;\alpha, \beta)+
                                       x(n-i+1)D^n_i(x;\alpha, \beta)\Big]\\
+\left((x-1)(i+1)C^{(\alpha,\beta)}_{ni}+
                    x(n-i+1)C^{(\alpha,\beta)}_{n,i+1}
+C^{(\alpha,\beta)}_{n+1,i+1}
          \frac{(n+\alpha+2)(n+\beta+2)}{(n+2)^{-1}(2n+\sigma+3)_2}\right)
                                                 R_{n+1}^{(\alpha,\beta)}(x)\\                                                 
+\frac{C^{(\alpha,\beta)}_{n+1,i+1}(n+2)}{(2n+\sigma+3)_2}
               \Big((n-i+1)(n+\alpha+2)-(i+1)(n+\beta+2)\Big)
                                              R_{n+2}^{(\alpha,\beta)}(x).
\end{multline*}

Applying twice the induction assumption to terms in square brackets and after
some algebra, we have
\begin{equation}\label{E:Proof-G-RecRel}
G^{(\alpha,\beta)}_{ni}
       \left(\xi_0(n)R^{(\alpha,\beta)}_n(x)+
                   \xi_1(n)R^{(\alpha,\beta)}_{n+1}(x)+
                         \xi_2(n)R^{(\alpha,\beta)}_{n+2}(x)\right),
\end{equation}
where the notation used is that of~\eqref{E:Def_xi_0}--\eqref{E:Def_xi_2}, and 
$$
G^{(\alpha,\beta)}_{ni}:=-C^{(\alpha,\beta)}_{ni}
               \frac{(n-i+1)(n-i+\alpha+1)-(i+1)(i+\beta+1)}
                    {2(n+1)(2n+\sigma+2)_2(n-i+\alpha+1)}.
$$

Indeed, it follows from~\eqref{E:JacobiRecRel} that~\eqref{E:Proof-G-RecRel} is
equal to zero. At the end, note the special case: in~\eqref{E:Proof-G-RecRel}, 
if $i=n-i$ and $\alpha=\beta$, both the expression in brackets and
$G_{ni}^{(\alpha, \beta)}$ are equal to zero.
\end{proof}

Let us stress that shifted Jacobi polynomials appearing 
in~\eqref{E:Def_T_I} (cf.~\eqref{E:Def_T_II}) do not depend on $i$. Now, solving
this first-order non-homogeneous recurrence relation and using~\eqref{E:D^n_0}
allows us to obtain---after some algebra---a more explicit formula for dual 
Bernstein polynomials.

\begin{corollary}
For $i=0,1,\ldots,n$, we have
\begin{multline*}
D^n_i(x;\alpha,\beta)=
  \binom{n}{i}^{-1}\frac{(-1)^{n-i}(\sigma+1)_n}
                       {K(\alpha+1)_n(\beta+1)_n}\left(\frac{x-1}{x}\right)^i\\
\times\left[
 R^{(\alpha,\beta+1)}_n(x)S^{(\alpha+1,\beta)}_{ni}\left(\frac{x}{x-1}\right)-
 R^{(\alpha+1,\beta)}_n(x)S^{(\alpha,\beta+1)}_{n,i-1}\left(\frac{x}{x-1}\right)
\right],                                              
\end{multline*}
where
$$
S^{(a,b)}_{mk}(z):=(b+1)_m
\sum_{j=0}^{k}\frac{(-m)_j(-m-a)_j}{j!(b+1)_j}z^j\qquad (0\leq k\leq m)
$$
(cf.~\eqref{E:JacobiP}).
\end{corollary}

\subsection{Homogeneous recurrence relations of order 2 and 3} 
                                                         \label{SS:RecRel-2-3}

Using Theorem~\ref{T:NH-Rec-Rel-1}, one can obtain a homogeneous relations of
order 2 and 3 by eliminating the non-homogeneity. The proofs are quite technical
therefore we omit them. Let us only mention that in the construction of
third-order recurrence relation for dual Bernstein polynomials we used the same
idea as in the proof of~\cite[Collorary 5.2]{ChW2018}, where recurrence of the
form~\eqref{E:DualBer-Rec-Rel-b} has been derived.

\begin{corollary}\label{C:RecRel-2}
Dual Bernstein polynomials satisfy the second-order recurrence relation of the
form
$$
u_0(i)D^n_i(x;\alpha,\beta)+u_1(i)D^n_{i+1}(x;\alpha,\beta)+
                     u_2(i)D^n_{i+2}(x;\alpha,\beta)=0\qquad (0\leq i\leq n-2),
$$
where
\begin{eqnarray*}
&&u_0(i):=(x-1)(i+1)(n-i+\alpha)T^{(\alpha,\beta)}_{n,i+1}(x),\\
&&u_1(i):=x(n-i)(n-i+\alpha)T^{(\alpha,\beta)}_{n,i+1}(x)+
          (x-1)(i+2)(i+\beta+2)T^{(\alpha,\beta)}_{ni}(x),\\
&&u_2(i):=x(n-i-1)(i+\beta+2)T^{(\alpha,\beta)}_{ni}(x),
\end{eqnarray*}
where the notation used is that of~\eqref{E:Def_T_I}.
\end{corollary}

The coefficients $u_j$ $(j=0,1,2)$ are not simple because they depend on two 
shifted Jacobi polynomials of degree $n$ in $x$. However, these polynomials 
are independent of $i$ and can be efficiently computed with the
recurrence~\eqref{E:JacobiRecRel} (cf.~Remark~\ref{R:JacobiRecRel}) and re-used
for all remaining $i$. Thus, Corollary~\ref{C:RecRel-2} may be useful in
numerical practice.

\begin{corollary}\label{C:RecRel-3}
For $0\leq i\leq n-3$, the polynomials $D^n_i(x;\alpha,\beta)$ satisfy the
following third-order recurence relation:
\begin{equation}\label{E:RecRel-3}
w_0(i)D^n_i(x;\alpha,\beta)+w_1(i)D^n_{i+1}(x;\alpha,\beta)+
        w_2(i)D^n_{i+2}(x;\alpha,\beta)+w_3(i)D^n_{i+3}(x;\alpha,\beta)=0.
\end{equation}
Here
\begin{eqnarray*}
&&w_0(i):=(x-1)(i+1)(n-i+\alpha-1)_2,\\
&&w_1(i):=(n-i+\alpha-1)[x(n-i)(n-i+\alpha)+2(x-1)(i+2)(i+\beta+2)],\\
&&w_2(i):=(i+\beta+2)[(x-1)(i+3)(i+\beta+3)+2x(n-i-1)(n-i+\alpha-1)],\\
&&w_3(i):=x(n-i-2)(i+\beta+2)_2.
\end{eqnarray*}
\end{corollary}

Notice that, compared to~\eqref{E:DualBer-Rec-Rel-b}, 
the recurrence~\eqref{E:RecRel-3} is simpler: i) it has lower order (third 
instead of fourth); ii) its coefficients $w_j$ ($0\leq j\leq 3$) are cubic
polynomials in $i$ (not quintic; cf.~\cite[Eq.~(4.2)]{ChW2018}).

\begin{remark}
Using three new recurrence relations given in \S\ref{SS:NH-RecRel-1} and
\S\ref{SS:RecRel-2-3} one can solve Problem~\ref{P:Problem1} with $O(n)$
computational complexity (cf.~Remark~\ref{R:JacobiRecRel}).
\end{remark}

\section{Algorithms for evaluating dual Bernstein polynomials}
                                                             \label{S:DB-Eval}

Let us come back to Problem~\ref{P:Problem1} of computing all $n+1$ dual
Bernstein polynomials of degree $n$ for fixed $n\in\mathbb N$, $\alpha,\beta>-1$
and $x\in[0,1]$. Recall that these polynomials are dual to Bernstein basis
polynomials~\eqref{E:BernPoly} in the interval $[0,1]$ (see~\eqref{E:InnerProd}
and Definition~\ref{D:DualBer}). So, in the context of applications of
polynomials $D^n_i(x;\alpha,\beta)$ presented in Section~\ref{S:Introduction},
the issue of their evaluation for $0\leq x\leq 1$ is the most important.  

If $x\in\{0,1\}$ then the value of the dual Bernstein polynomial can be easily
obtained (cf.~\eqref{E:DualBer-1} or~\eqref{E:DualBer-0}). Now, suppose that
$x\in(0,1)$. In this section we propose algorithms for evaluating polynomials 
$D^n_i(x;\alpha,\beta)$ $(0\leq i\leq n)$ using the first-order non-homogeneous
recurrence relation (cf.~Theorem~\ref{T:NH-Rec-Rel-1}).  

To obtain accurate methods, it is necessary to be mindful of numerical
difficulties arising when recurrent computations are performed.
See~\cite{Wimp1984}. This is the reason that, in the sequel, we consider two 
ways of using relation~\eqref{E:NH-Rec-Rel-1}, i.e., with a \textit{forward} 
and a \textit{backward} direction of computations.

For a fixed $0\leq i\leq n$, let us define a \textit{forward computation} of
$D^n_i(x;\alpha,\beta)$ as a computation that, starting from 
$D^n_0(x;\alpha,\beta)$, computes $D^n_i(x;\alpha,\beta)$ using 
Theorem~\ref{T:NH-Rec-Rel-1}. Analogously, a~\textit{backward computation} of 
$D^n_i(x; \alpha, \beta)$ starts with $D^n_n(x;\alpha,\beta)$ and uses 
Theorem~\ref{T:NH-Rec-Rel-1} as well. As mentioned before, some numerical
difficulties may arise when performing these computations --- especially for
sufficiently large $n$ and $i$. One can mitigate this issue by performing,
for certain parameter $J\in\mathbb N$ ($0\leq J\leq n$), a forward computation 
of $D^n_i(x;\alpha,\beta)$ for $i=0,1,\dots,J$ and a backward computation of
$D^n_i(x;\alpha,\beta)$ for $i=J+1,J+2,\dots,n$. Note that,
using~\eqref{E:DualBerSym}, a~backward computation can be expressed as a forward
computation with changed parameters.

We have found that, in order to determine the value of $J$, one can use the
function
\begin{equation}\label{E:optimalJ}
J\equiv J(n,x)=\mbox{round}\left(n\cdot p(x)\right),
\end{equation}
where $p$ is a cubic polynomial in $x$ which satisfies the following
interpolation conditions:
\begin{center}
\begin{tabular}{c|c|c|c|c}
 $x$    & 0.01 & 0.3 & 0.7 & 0.99\\ \hline
 $p(x)$ & 0.1  & 0.4 & 0.6 & 0.9
\end{tabular},
\end{center}
and $\mbox{round}(z)$ denotes the nearest integer to the real number $z$. 
It can be checked that
\begin{eqnarray*}
p(x)&=&1.58084223194525186\ldots\cdot x^3-2.37126334791787779\ldots\cdot x^2\\
     &&+1.62239798468112882\ldots\cdot x+0.08401156564574855\ldots.     
\end{eqnarray*}

Let us stress that such choice of $J$ has been established experimentally and 
is used in all algorithms, as well as numerical tests, presented in this
paper.

\subsection{Algorithms}

For given $n\in\mathbb N$ and $\alpha,\beta>-1$, an implementation of a forward
computation of $D^n_i(x; \alpha, \beta)$ for $i=0,1,\ldots,j$ and a fixed $0\leq
j\leq n$ at one point $x\in(0,1)$ is presented in Algorithm~\ref{A:ForwardOneX}.

\begin{algorithm}[ht!]
\caption{Computation of $j+1$ first dual Bernstein polynomials of degree $n$ at
point $x\in(0,1)$}\label{A:ForwardOneX}
\begin{algorithmic}[1]
\Procedure {DualBer}{$n,\alpha,\beta,x,j,K$}

\State $\alpha1 \gets \alpha+1,\ \beta1 \gets \beta+1$

\State $n1\gets n+\alpha1,\ x1x \gets (x-1)/x$

\State $C \gets (-1)^{n+1}\cdot K/n1\cdot 
                   \prod_{j=0}^{n-1}(1+\beta1/(j+\alpha1))$

\State\label{A:L1} $R1 \gets n1\cdot R^{(\alpha,\beta1)}_n(x)$

\State\label{A:L2} $R2 \gets x1x\cdot (n+\beta1)\cdot R^{(\alpha1,\beta)}_n(x)$

\State $D[0] \gets -C\cdot R1$

\For {$i \gets 1,j$}
    \State $p \gets i-n-1$

    \State $q \gets i/p$

    \State $C \gets C\cdot (p-\alpha1)/(i+\beta)$
    
    \State $D[i] \gets q\cdot x1x\cdot D[i-1] - 
                            C\cdot (R1+q\cdot R2)$
\EndFor

\State \Return $D$

\EndProcedure
\end{algorithmic}
\end{algorithm}

For fixed $n\in\mathbb N$ and $\alpha,\beta>-1$,
Algorithm~\ref{A:DualBernsteinOneX} computes the values of all $n+1$ dual
Bernstein polynomials of degree $n$ at one point $x\in(0,1)$. It computes the
value $J$ (cf.~\eqref{E:optimalJ}) and then performs two forward computations,
utilizing Algorithm~\ref{A:ForwardOneX}. This algorithm returns an array
$D\equiv D[0..n]$, where
$$
D[i]=D^n_i(x;\alpha,\beta)\qquad (0\leq i\leq n).
$$

\begin{remark}\label{R:DualBernsteinOneX}
Shifted Jacobi polynomials $R_n^{(\alpha,\beta+1)}$ and
$R_n^{(\alpha+1,\beta)}$ (cf.~lines \ref{A:L1}, \ref{A:L2} in
Algorithm~\ref{A:ForwardOneX}) can be evaluated using recurrence 
relation~\eqref{E:JacobiRecRel} (cf.~Remark~\ref{R:JacobiRecRel}) or even
explicit formula~\eqref{E:JacobiP}. Thus, the computational complexity of
Algorithm~\ref{A:DualBernsteinOneX} is $O(n)$.
\end{remark}

\begin{algorithm}[ht!]
\caption{Computation of all $n+1$ dual Bernstein polynomials of degree $n$ at
point $x\in(0,1)$}\label{A:DualBernsteinOneX}
\begin{algorithmic}[1]
\Procedure {AllDualBer}{$n,\alpha,\beta,x$}

\State $\alpha1 \gets \alpha+1,\ \beta1 \gets \beta+1$

\State $K \gets \Gamma(\alpha1+\beta1)/\left(\Gamma(\alpha1)\cdot 
                                             \Gamma(\beta1)\right)$

\State $J \gets J(n,x)$

\State $D[0..J] \gets \mbox{\textsc{DualBer}}(n,\alpha,\beta,x,J,K)$

\State $D[J+1..n] \gets \mbox{\textsc{ReverseArray}}
                 (\mbox{\textsc{DualBer}}(n,\beta,\alpha,1-x,n-J-1,K))$

\State \Return $D$

\EndProcedure
\end{algorithmic}
\end{algorithm}

Note that the quantities $q$ and $C$ in Algorithm~\ref{A:ForwardOneX}, as well 
as the quantity $K$ in Algorithm~\ref{A:DualBernsteinOneX}, are independent of
$x$. They can be, therefore, computed once for given $n\in\mathbb N$,
$\alpha,\beta>-1$ and used across multiple instances of Problem~\ref{P:Problem1}
for different values of $x\in(0,1)$. This approach is realized in
Algorithms~\ref{A:ForwardManyX} and~\ref{A:DualBernsteinManyX}. Note that they
require $O(n)$ additional memory to store $C$ and $q$. 

\begin{algorithm}[ht!]
\caption{Computation of $j+1$ first dual Bernstein polynomials of degree $n$ at
point $x\in(0,1)$ --- with preprocessing}\label{A:ForwardManyX}
\begin{algorithmic}[1]
\Procedure {DualBer-2}{$n,\alpha,\beta,x,j,q,C$}

\State $\alpha1 \gets \alpha+1,\ \beta1 \gets \beta+1$

\State $n1\gets n+\alpha1,\ x1x \gets (x-1)/x$

\State $R1 \gets n1\cdot R^{(\alpha,\beta1)}_n(x)$

\State $R2 \gets x1x\cdot (n+\beta1)\cdot R^{(\alpha1,\beta)}_n(x)$

\State $D[0] \gets C[0]\cdot R1/n1$

\For {$i \gets 1,j$}

    \State $D[i] \gets q[i-1]\cdot x1x\cdot D[i-1] - 
                         C[i]\cdot(R1+q[i-1]\cdot R2)$
\EndFor

\State \Return $D$

\EndProcedure
\end{algorithmic}
\end{algorithm}

\begin{algorithm}[ht!]
\caption{Computation of all $n+1$ dual Bernstein polynomials of degree $n$ at
multiple points $x_0,x_1,\ldots,x_M\in(0,1)$}\label{A:DualBernsteinManyX}
\begin{algorithmic}[1]
\Procedure {AllDualBer-2}{$n,\alpha,\beta,[x_0,x_1,\ldots,x_M]$}

\State $\alpha1 \gets \alpha+1,\ \beta1 \gets \beta+1$

\State $K \gets \Gamma(\alpha1+\beta1)/\left(\Gamma(\alpha1)\cdot 
                                             \Gamma(\beta1)\right)$

\State $q[0] \gets -1/n$                                             

\State $C[0] \gets (-1)^n\cdot K\cdot 
                 \prod_{j=0}^{n-1}(1+\beta1/(j+\alpha1))$
                 
\State $C[1] \gets C[0]/\beta1$

\For {$i \gets 1,n-1$}
  
    \State $p \gets i-n$

    \State $q[i] \gets (i+1)/p$

    \State $C[i+1] \gets C[i]\cdot (p-\alpha1)/(i+\beta1)$    
\EndFor

\For {$m \gets 0,M$}
        
    \State $J \gets J(n,x_m)$
    
    \State $D[m,0..J] \gets \mbox{\textsc{DualBer-2}}
                                            (n,\alpha,\beta,x_m,J,q,C)$
    
    \State $D[m,J+1..n] \gets \mbox{\textsc{ReverseArray}}
             (\mbox{\textsc{DualBer-2}}(n,\beta,\alpha,1-x_m,n-J-1,q,C))$
    
\EndFor

\State \Return $D$

\EndProcedure
\end{algorithmic}
\end{algorithm}

After executing Algorithm~\ref{A:DualBernsteinManyX}, we obtain 
a two-dimensional array $D\equiv D[0..M,0..n]$, where
$$
D[m,i]=D^n_i(x_m;\alpha,\beta)\qquad (0\leq m\leq M;\ 0\leq i\leq n).
$$
The computational complexity of this algorithm is $O(nM)$ 
(cf.~Remark~\ref{R:DualBernsteinOneX}).

\section{Numerical experiments}                       \label{S:NumExperiments}

The algorithms presented in the previous section have been tested for numerical
stability. The computations have been performed in the computer algebra 
system \textsf{Maple{\small\texttrademark}14}---using single
(\texttt{Digits:=8}), double (\texttt{Digits:=18}) and quadruple
(\texttt{Digits:=32}) precision---on a computer with \texttt{Intel(R) Core(TM)
i5-2540M CPU @ 2.60GHz} processor and \texttt{4 GB} of \texttt{RAM}.

We measure the \textit{accuracy} of approximation $\widetilde{v}$ of a nonzero
number $v$ by computing the quantity
\begin{equation}\label{E:Acc}
\mbox{\texttt{acc}}(\widetilde{v},v):=
               -\log_{10}\left|1-\frac{\widetilde{v}}{v}\right|.
\end{equation}
Hence, $\mbox{\texttt{acc}}(\widetilde{v},v)$ is the number of exact significant 
decimal digits (\texttt{acc} in short) in the approximation $\widetilde{v}$ of
the number $v$.

\begin{table}[!ht]
\renewcommand{\arraystretch}{1.25}
\begin{center}
\begin{tabular}{cc|c|c|c}
         & & \texttt{Digits:=8} & \texttt{Digits:=18} & \texttt{Digits:=32} \\
          \hline
           & $\alpha=\beta=0$          & 7.64 & 17.67 & 31.80 \\
  $n=10$   & $\alpha=\beta=-0.5$       & 6.97 & 17.03 & 31.04 \\
           & $\alpha=-0.33, \beta=5.6$ & 7.14 & 17.39 & 31.27 \\
           \hline
           & $\alpha=\beta=0$          & 7.24 & 17.15 & 31.14 \\
  $n=20$   & $\alpha=\beta=-0.5$       & 6.72 & 16.82 & 30.95 \\
           & $\alpha=-0.33, \beta=5.6$ & 6.65 & 17.47 & 31.17 \\
           \hline
           & $\alpha=\beta=0$          & 6.77 & 17.58 & 30.46 \\
  $n=50$   & $\alpha=\beta=-0.5$       & 6.58 & 17.47 & 30.55 \\
           & $\alpha=-0.33, \beta=5.6$ & 7.00 & 17.43 & 30.94 \\
           \hline
           & $\alpha=\beta=0$          & 6.98 & 16.80 & 30.32 \\
  $n=100$  & $\alpha=\beta=-0.5$       & 6.46 & 17.06 & 31.00 \\
           & $\alpha=-0.33, \beta=5.6$ & 6.79 & 17.30 & 31.16 \\
           \hline
           & $\alpha=\beta=0$          & 7.28 & 16.56 & 31.18 \\
  $n=200$  & $\alpha=\beta=-0.5$       & 6.41 & 16.12 & 30.65 \\
           & $\alpha=-0.33, \beta=5.6$ & 6.21 & 17.02 & 31.00 \\
           \hline
           & $\alpha=\beta=0$          & 6.65 & 17.01 & 30.95 \\
  $n=500$  & $\alpha=\beta=-0.5$       & 6.08 & 16.36 & 30.70 \\
           & $\alpha=-0.33, \beta=5.6$ & 6.13 & 16.80 & 30.91 \\
           \hline
           & $\alpha=\beta=0$          & 6.51 & 16.31 & 30.23 \\
  $n=1000$ & $\alpha=\beta=-0.5$       & 6.23 & 16.56 & 29.99 \\
           & $\alpha=-0.33, \beta=5.6$ & 5.85 & 15.73 & 29.73 \\
           \hline
           & $\alpha=\beta=0$          & 6.09 & 16.88 & 29.64 \\
  $n=2000$ & $\alpha=\beta=-0.5$       & 6.87 & 16.43 & 29.56 \\
           & $\alpha=-0.33, \beta=5.6$ & 6.21 & 15.81 & 30.43 \\
           \hline
           & $\alpha=\beta=0$          & 5.57 & 15.36 & 30.12 \\
  $n=5000$ & $\alpha=\beta=-0.5$       & 5.62 & 15.45 & 29.57 \\
           & $\alpha=-0.33, \beta=5.6$ & 5.29 & 15.42 & 30.51 \\
           \hline
\end{tabular}
\caption{Mean number of \texttt{acc} (cf.~\eqref{E:Acc}) obtained by using
Algorithm~\ref{A:DualBernsteinManyX} for
$x\in\{0.01,0.02,\dots,0.99\}$.}\label{T:Mean}
\end{center}
\renewcommand{\arraystretch}{1}
\end{table}

For fixed $n\in\mathbb N$ and $\alpha,\beta>-1$, the experiments involved
computing values of all $n+1$ dual Bernstein polynomials of degree $n$ for
$x\in\{0.01,0.02,\dots,0.99\}$ using Algorithm~\ref{A:DualBernsteinManyX}, where
\textsf{Maple{\small\texttrademark}14} \texttt{GAMMA} and \texttt{JacobiP}
procedures have been used to compute values of $\Gamma$ function and shifted
Jacobi polynomials, respectively. For each of $(n+1)\cdot 99$ obtained values, 
we have computed the number of exact significant decimal digits
(cf.~\eqref{E:Acc}), where results computed by the same algorithm but in 
a 512-digit arithmetic (\texttt{Digits:=512}) have been assumed to be accurate
while comparing to these done for \texttt{Digits:=8,18,32}.

\begin{table}[!ht]
\renewcommand{\arraystretch}{1.25}
\begin{center}
\begin{tabular}{cc|c|c|c}
 & & \texttt{Digits:=8} & \texttt{Digits:=18} & \texttt{Digits:=32} \\\hline
           & $\alpha=\beta=0$          & 6.34 & 16.36 & 30.47 \\
  $n=10$   & $\alpha=\beta=-0.5$       & 6.01 & 16.18 & 30.41 \\
           & $\alpha=-0.33, \beta=5.6$ & 6.31 & 16.26 & 30.35 \\
           \hline
           & $\alpha=\beta=0$          & 6.12 & 16.23 & 30.25 \\
  $n=20$   & $\alpha=\beta=-0.5$       & 5.99 & 16.16 & 30.32 \\
           & $\alpha=-0.33, \beta=5.6$ & 6.13 & 16.30 & 30.18 \\
           \hline
           & $\alpha=\beta=0$          & 6.31 & 16.44 & 30.16 \\
  $n=50$   & $\alpha=\beta=-0.5$       & 6.22 & 16.28 & 30.26 \\
           & $\alpha=-0.33, \beta=5.6$ & 6.31 & 16.31 & 30.34 \\
           \hline
           & $\alpha=\beta=0$          & 6.32 & 16.38 & 30.15 \\
  $n=100$  & $\alpha=\beta=-0.5$       & 6.17 & 16.36 & 30.36 \\
           & $\alpha=-0.33, \beta=5.6$ & 6.28 & 16.27 & 30.23 \\
           \hline
           & $\alpha=\beta=0$          & 6.11 & 16.13 & 30.21 \\
  $n=200$  & $\alpha=\beta=-0.5$       & 6.05 & 15.89 & 30.11 \\
           & $\alpha=-0.33, \beta=5.6$ & 5.96 & 16.17 & 30.15 \\
           \hline
           & $\alpha=\beta=0$          & 6.17 & 16.18 & 30.15 \\
  $n=500$  & $\alpha=\beta=-0.5$       & 5.90 & 16.02 & 30.17 \\
           & $\alpha=-0.33, \beta=5.6$ & 5.94 & 16.03 & 30.16 \\
           \hline
           & $\alpha=\beta=0$          & 6.05 & 16.02 & 29.93 \\
  $n=1000$ & $\alpha=\beta=-0.5$       & 5.87 & 16.11 & 29.82 \\
           & $\alpha=-0.33, \beta=5.6$ & 5.74 & 15.61 & 29.63 \\
           \hline
           & $\alpha=\beta=0$          & 5.84 & 16.05 & 29.54 \\
  $n=2000$ & $\alpha=\beta=-0.5$       & 6.02 & 16.01 & 29.42 \\
           & $\alpha=-0.33, \beta=5.6$ & 5.86 & 15.61 & 29.96 \\
           \hline
           & $\alpha=\beta=0$          & 5.46 & 15.30 & 29.65 \\
  $n=5000$ & $\alpha=\beta=-0.5$       & 5.49 & 15.36 & 29.46 \\
           & $\alpha=-0.33, \beta=5.6$ & 5.24 & 15.35 & 29.82 \\
           \hline
 \end{tabular}
\caption{First percentile number of \texttt{acc} (cf.~\eqref{E:Acc}) obtained by
using Algorithm~\ref{A:DualBernsteinManyX} for
$x\in\{0.01,0.02,\dots,0.99\}$.}\label{T:1Percent}
\end{center}
\renewcommand{\arraystretch}{1}
\end{table}

The experiments have been performed for dual Bernstein polynomials of degrees
$n\in\{10,20,50,100,200,500,1000,2000,5000\}$ and three $\alpha,\beta$ 
choices --- \textit{Legendre's} ($\alpha=\beta=0$), \textit{Chebyshev's}
($\alpha=\beta=-0.5$), and a \textit{non-standard} choice ($\alpha=-0.33$,
$\beta=5.6$). A mean (Table~\ref{T:Mean}), first percentile 
(Table~\ref{T:1Percent}) and minimal (Table~\ref{T:Min}) number of exact
significant decimal digits have been computed.

The numerical results show that the proposed method for evaluating dual 
Bernstein polynomials works very well even for large degrees. Note that results
given in Tables~\ref{T:Mean} and~\ref{T:1Percent} are almost the same. It
indicates that at least $99\%$ of obtained values have greater or similar number
of exact significant decimal digits than these given in Table~\ref{T:Mean}, thus
making the presented algorithms useful (for example, in numerical evaluation of
integrals~\eqref{E:IntI_k}, even for large $n$). Even though in
\textit{pessimistic cases} the algorithms lose a significant amount of
precision (especially for \texttt{Digits:=8}; see Table~\ref{T:Min}), they 
happen rarely (compare with Table~\ref{T:1Percent}) and do not significantly
affect the average number of correct decimal digits (cf.~Table~\ref{T:Mean}).

\begin{table}[!ht]
\renewcommand{\arraystretch}{1.25}
\begin{center}
\begin{tabular}{cc|c|c|c}
 & & \texttt{Digits:=8} & \texttt{Digits:=18} & \texttt{Digits:=32} \\\hline
           & $\alpha=\beta=0$          & 4.09 & 14.69 & 28.96 \\
  $n=10$   & $\alpha=\beta=-0.5$       & 4.97 & 15.36 & 29.06 \\
           & $\alpha=-0.33, \beta=5.6$ & 4.93 & 14.96 & 29.09 \\
           \hline
           & $\alpha=\beta=0$          & 5.33 & 15.35 & 29.41 \\
  $n=20$   & $\alpha=\beta=-0.5$       & 4.86 & 14.87 & 29.56 \\
           & $\alpha=-0.33, \beta=5.6$ & 5.39 & 15.12 & 28.91 \\
           \hline
           & $\alpha=\beta=0$          & 4.83 & 14.65 & 28.75 \\
  $n=50$   & $\alpha=\beta=-0.5$       & 4.47 & 14.32 & 28.48 \\
           & $\alpha=-0.33, \beta=5.6$ & 4.65 & 14.13 & 29.16 \\
           \hline
           & $\alpha=\beta=0$          & 4.47 & 14.84 & 28.73 \\
  $n=100$  & $\alpha=\beta=-0.5$       & 4.36 & 14.37 & 28.48 \\
           & $\alpha=-0.33, \beta=5.6$ & 2.98 & 12.84 & 27.35 \\
           \hline
           & $\alpha=\beta=0$          & 3.41 & 13.54 & 27.19 \\
  $n=200$  & $\alpha=\beta=-0.5$       & 3.62 & 13.42 & 27.73 \\
           & $\alpha=-0.33, \beta=5.6$ & 4.17 & 14.31 & 28.04 \\
           \hline
           & $\alpha=\beta=0$          & 3.15 & 13.65 & 27.06 \\
  $n=500$  & $\alpha=\beta=-0.5$       & 2.01 & 12.28 & 26.47 \\
           & $\alpha=-0.33, \beta=5.6$ & 3.37 & 13.37 & 27.26 \\
           \hline
           & $\alpha=\beta=0$          & 2.92 & 12.97 & 27.30 \\
  $n=1000$ & $\alpha=\beta=-0.5$       & 3.21 & 13.41 & 27.56 \\
           & $\alpha=-0.33, \beta=5.6$ & 3.18 & 12.99 & 27.44 \\
           \hline
           & $\alpha=\beta=0$          & 2.16 & 12.24 & 26.03 \\
  $n=2000$ & $\alpha=\beta=-0.5$       & 1.46 & 11.85 & 25.40 \\
           & $\alpha=-0.33, \beta=5.6$ & 2.11 & 11.93 & 25.97 \\
           \hline
           & $\alpha=\beta=0$          & 0 & 5.38 & 18.85 \\
  $n=5000$ & $\alpha=\beta=-0.5$       & 0 & 4.25 & 18.41 \\
           & $\alpha=-0.33, \beta=5.6$ & 0 & 5.64 & 19.33 \\
           \hline
 \end{tabular}
\caption{Minimal number of \texttt{acc} (cf.~\eqref{E:Acc}) obtained by using
Algorithm~\ref{A:DualBernsteinManyX} for
$x\in\{0.01,0.02,\dots,0.99\}$.}\label{T:Min}
\end{center}
\renewcommand{\arraystretch}{1}
\end{table}

\bibliographystyle{elsarticle-num} 
\biboptions{sort&compress}
\bibliography{dualB-eval}


\end{document}